\documentclass[12]{amsart}
\usepackage{mathrsfs}
\usepackage{stmaryrd}
\usepackage[latin1]{inputenc}
\usepackage{amsmath}
\usepackage{amsfonts}
\usepackage{amssymb}
\usepackage[all,cmtip]{xy}
\usepackage{verbatim}   

\usepackage{amsthm}

\theoremstyle{plain}
\newtheorem{thm}{Theorem}[subsection]

\newtheorem{prop}[thm]{Proposition}

\newtheorem{lemma}[thm]{Lemma}

\theoremstyle{definition}
\newtheorem{rem}[thm]{Remark}

\title{A Note on Potential Diagonalizability of Crystalline Representations}

\author{HUI GAO, TONG LIU}

\address{Beijing International Center for Mathematical Research, Peking University, No. 5 Yiheyuan Road, Haidian District, Beijing 100871, China }
 \email{gaohui@math.pku.edu.cn}

\address{Department of Mathematics, Purdue University, 150 N. University Street, West Lafayette, IN 47907, USA}
\email{tongliu@math.purdue.edu}


\begin{document}
\maketitle

\pagestyle{myheadings}
\markright{A Note on Potential Diagonalizability of Crystalline Representations}

\newcommand{\Zp}{\mathbb{Z}_p}
\newcommand{\Qp}{\mathbb{Q}_p}
\newcommand{\Z}{\mathbb{Z}}
\newcommand{\Q}{\mathbb{Q}}
\newcommand{\R}{\mathbb{R}}
\newcommand{\C}{\mathbb{C}}
\newcommand{\N}{\mathbb{N}}

\newcommand{\Qpbar}{\overline{\Q}_p}

\newcommand{\cris}{\textnormal{cris}}
\newcommand{\dR}{\textnormal{dR}}
\newcommand{\Spec}{\textnormal{Spec}}
\newcommand{\Hom}{\textnormal{Hom}}
\newcommand{\Frac}{\textnormal{Frac}}
\newcommand{\Fil}{\textnormal{Fil}}
\newcommand{\Mod}{\textnormal{Mod}}
\newcommand{\ModFI}{\textnormal{ModFI}}
\newcommand{\Rep}{\textnormal{Rep}}
\newcommand{\Gal}{\textnormal{Gal}}
\newcommand{\Ker}{\textnormal{Ker}}
\newcommand{\GL}{\textnormal{GL}}
\newcommand{\HT}{\textnormal{HT}}

\newcommand{\Zl}{\mathbb{Z}_{\ell}}
\newcommand{\Ql}{\mathbb{Q}_{\ell}}
\newcommand{\A}{\mathbb{A}}

\newcommand{\bigM}{\mathcal{M}}
\newcommand{\bigD}{\mathcal{D}}
\newcommand{\huaS}{\mathfrak{S}}
\newcommand{\huaM}{\mathfrak{M}}
\newcommand{\Acris}{A_{\textnormal{cris}}}
\newcommand{\bigO}{\mathcal{O}}

\newcommand{\bigN}{\mathcal{N}}
\newcommand{\huaN}{\mathfrak{N}}
\newcommand{\m}{\mathfrak{m}}

\newcommand{\M}{\mathcal{M}}
\newcommand{\BdR}{B_{\textnormal{dR}}}
\newcommand{\bigOE}{\mathcal{O}_{\mathcal{E}}}

\newcommand{\bolda}{\boldsymbol{\alpha}}
\newcommand{\bolde}{\boldsymbol{e}}

\newcommand{\boldd}{\boldsymbol{d}}

\newcommand{\tildeS}{\tilde{S}}
\newcommand{\tS}{\tilde{S}}
\newcommand{\tildeM}{\tilde{M}}
\newcommand{\tM}{\tilde{M}}

\newcommand{\Asthat}{\widehat{A_{\textnormal{st}}}}
\newcommand{\D}{\mathcal{D}}

\newcommand{\MF}{MF^{(\varphi, N)}}
\newcommand{\MFwa}{MF^{(\varphi, N)-\textnormal{w.a.}}}
\newcommand{\bigMF}{\mathcal{MF}^{(\varphi, N)}}
\newcommand{\bigMFwa}{\mathcal{MF}^{(\varphi, N)-\textnormal{w.a.}}}


\begin{abstract}
 Let $K_0/\Qp$ be a finite unramified extension, $G_{K_0}$ the Galois group $\text{Gal}(\Qpbar/K_0)$. We show that all crystalline representations of $G_{K_0}$ with Hodge-Tate weights $\subseteq\{0, \ldots, p-1\}$ are potentially diagonalizable.
\end{abstract}
\footnote{The second author  is partially supported by NSF grant DMS-0901360.}
\tableofcontents
\section{Introduction}
Let $p$ be a prime,  $K$ a finite  extension over $\Q_p$ and $G_{K}$ the absolute Galois group $\text{Gal}(\Qpbar/K)$. In \cite{BLGGT10} \S1.4, \emph{potential diagonalizability} is defined for a potentially crystalline representation of $G_{K}$. Since potential diagonalizability is the local condition at $p$ for a global Galois representation in the automorphy lifting theorems proved in \cite{BLGGT10} (cf. Theorem {\bf B, C}), it is quite interesting to investigate what kind of potentially crystalline representations are indeed potentially diagonalizable. Let $K_0$ be a finite unramified extension of $\Q_p$.
By using Fontaine-Laffaille's theory, Lemma 1.4.3 (2) in \cite{BLGGT10}  proved that any crystalline representation of $G_{K_0}$ with Hodge-Tate weights in $\{0, \dots, p-2\}$ is potentially diagonalizable.

In this short note, we show that the idea in \cite{BLGGT10} can be extended to prove the potential diagonalizability of crystalline representations of $G_{K_0}$ with Hodge-Tate weights in $\{0 , \dots , p-1\}$. Let $\rho: G_{K_0} \to \GL_d (\Qpbar)$ be a crystalline representation with Hodge-Tate weights in $\{0, \dots , p-1\}$. To prove the potential diagonalizability of $\rho$, we first reduce to the case that $\rho$ is irreducible. Then $\rho$ is  \emph{nilpotent} (see definition in \S2.2). Note that  Fontaine-Laffaille's theory can be extended to nilpotent representations. Hence we can follow the similar idea in \cite{BLGGT10} to conclude the potential diagonalizability of $\rho$.

\

{\bf Acknowledgement: } It is a pleasure to thank David Geraghty and Toby Gee for very useful conversations and correspondence. We also would like to thank the anonymous referee for helping to improve the exposition.

\section*{Notations}
Throughout this note, $K$ is always a finite extension of $\Q_p$ with the absolute Galois group $G_{K}:= \Gal(\Qpbar/K)$. Let $K_0$ be a  finite unramified extension of $\Qp$ with residue field $k$. We denote $W(k)$ its ring of integers and $\text{Frob}_{W(k)}$ the arithmetic Frobenius on $W(k)$. If  $E$ is a finite extension of $\Qp$ then we write $\bigO$ the ring of integers, $\varpi$ its uniformizer and $\mathbb{F}= \bigO/\varpi\bigO$ its residue field. If $A$ is a local ring, we denote $\m_A$ the maximal ideal of $A$ and equip $A$ with the $\m_A$-adic topology. Let $\rho: G_K\to \GL_d (A)$ be a continuous representation with the ambient space $M = \oplus_{i=1}^d A$. We always denote $\rho^*$ the dual representation induced by $\Hom_{A}(M , A)$. Let $\rho: G_K\to \GL_d (\Qpbar)$ be a de Rham representation of $G_K$. Then  $D_\dR(\rho^*)$ is a filtered $K \otimes_{\Q_p}\Qpbar$-module. For any embedding $\tau : K \to \Qpbar$, we define the set of $\tau$-Hodge-Tate weights
$$\HT_\tau (\rho):  = \{i \in \Z| \text{gr}^i (D_\dR (\rho^*)) \otimes_{K \otimes_{\Q_p} \Qpbar} (K \otimes_{K, \tau} \Qpbar) \not = 0\}.$$
In particular, if $\epsilon$ is the $p$-adic cyclotomic character then $\HT_\tau (\epsilon)= \{1\}$ (here our convention is slightly different from that in \cite{BLGGT10}).

\newcommand{\rhobar}{\overline{\rho}}
\newcommand{\PD}{{potentially diagonalizable}}

\section{Definitions and Preliminaries}

\subsection{Potential Diagonalizability}

\newcommand{\bigOp}{\bigO_{\overline{\Q}_p}}
\newcommand{\bigR}{R^{\box}_{\overline{\rho_1}, \HT_{\tau}(\rho_1), K'-cris} \otimes \Qpbar }

 We recall the definition of potential diagonalizability from \cite{BLGGT10}. Given two continuous representations $\rho_1, \rho_2: G_K \to \GL_d(\bigOp)$, we say that $\rho_1$ \emph{connects to} $\rho_2$, denoted by $\rho_1 \sim \rho_2$,  if:

\begin{itemize}
\item the two reductions $\bar \rho_i:= \rho_i  \bmod \mathfrak{m}_{\bigOp}$ are equivalent to each other;
\item both $\rho_1$ and $\rho_2$ are potentially crystalline;
\item for each embedding $\tau: K \hookrightarrow \overline{\Q}_p$, we have $\HT_{\tau}(\rho_1)=\HT_{\tau}(\rho_2)$;
\item $\rho_1$ and $\rho_2$ define points on the same irreducible component of the scheme  $\Spec(R^{\Box}_{\bar \rho_1, \{\HT_{\tau}(\rho_1)\}, K'\textnormal{-cris}}[\frac 1 p] )$ for some sufficiently large field extension $K'/K$. Here $R^{\Box}_{\bar \rho_1, \{\HT_{\tau}(\rho_1)\}, K'\textnormal{-cris}}$ is the quotient of the framed universal deformation ring $R^{\Box}_{\bar \rho_1}$ corresponding to liftings $\rho$ with $\HT_{\tau}(\rho)=\HT_{\tau}(\rho_1)$ for all $\tau$ and with $\rho \mid_{G_{K'}}$ crystalline. The existence of $R^{\Box}_{\bar \rho_1, \{\HT_{\tau}(\rho_1)\}, K'\textnormal{-cris}}$ is the main result of \cite{Kis08}.
\end{itemize}

Clearly the relation $\sim$ is an equivalence relation. A representation $\rho: G_K \to \GL_d(\bigOp)$ is called \emph{diagonalizable} if it is crystalline and connects to a sum of crystalline characters $\chi_1 \oplus \cdots \oplus \chi_d$. It is called \emph{potentially diagonalizable} if $\rho \mid_{G_{K'}}$ is diagonalizable for some finite extension $K'/K$.

\begin{rem}
By Lemma 1.4.1 of \cite{BLGGT10},  the potential diagonalizability is well defined for a representation $\rho: G_K \to \GL_d(\Qpbar)$ because for any two $G_K$-stable $\bigO_{\Qpbar}$-lattices $L$ and $L'$,  $L$ is potentially diagonalizable if and only if $L'$ is potentially diagonalizable.
\end{rem}

\begin{lemma}\label{reduce to irreducible}
Suppose $\rho: G_K \to \GL_d (\Qpbar) $ is potentially crystalline. Let $\Fil^i$ be a $G_K$-invariant filtration on $\rho$.  Then $\rho$ is \PD \ if and only if $\oplus_i \textnormal{gr}^i\rho$ is \PD.
\end{lemma}
\begin{proof} We can always choose a $G_K$-stable $\bigO_{\Qpbar}$-lattice $M$ inside the ambient space of $\rho$ such that $\Fil^i \rho \cap M$ is an $\bigO_{\Qpbar}$-summand of $M$ and the reduction $\bar M$ is semi-simple. Then the lemma follows item (7) of the numbered list preceding Lemma 1.4.1 of \cite{BLGGT10}.
\end{proof}

\subsection{Nilpotency and Fontaine-Laffaille Data}

\newcommand{\bigMFO}{\mathcal{MF}_{\bigO}}
\newcommand{\MFtor}{\mathcal{MF}_{\bigO,  \textnormal{tor}}}
\newcommand{\MFfree}{\mathcal{MF}_{\bigO,  \textnormal{fr}}}
\newcommand{\repcris} {\textnormal{Rep}_{E, \textnormal{cris}}^{[0, p-1]}(G_{K_0})}
\newcommand{\repcrisO} {\textnormal{Rep}_{\bigO, \textnormal{cris}}^{[0, p-1]}(G_{K_0})}
\newcommand{\repcrisOu} {\textnormal{Rep}_{\bigO, \textnormal{cris}}^{[0, p-1], \n}(G_{K_0})}
\newcommand{\n} {\textnormal{n}}

Let $E$ be a finite extension of $\Q_p$. Recall that we write $\bigO$ the ring of integers, $\varpi$ its uniformizer and $\mathbb{F}= \bigO/\varpi\bigO$ its residue field. Write $W(k)_{\bigO}:= W(k) \otimes_{\Z_p} \bigO$. By imitating \cite{FL82} \S7.7,
let $\bigMFO$ denote the category of finitely generated $W(k)_\bigO$-modules $M$ with
\begin{itemize}
\item a decreasing filtration $\Fil^i M$ by $W(k)_\bigO$-submodules which are $W(k)$-direct summands, where $\Fil^0 M=M$ and $\Fil^p M=\{0\}$;
\item $\textnormal{Frob}_{W(k)}\otimes 1$-semi-linear and $1 \otimes \bigO$-linear maps $\varphi_i : \Fil^i M \to M$ with $\varphi_i\mid_{\Fil^{i+1}M} = p\varphi_{i+1}$ and $\sum_{i=0}^{p-1} \varphi_i(\Fil^iM)=M$.
\end{itemize}
The morphisms in $\bigMFO$ are $W(k)_\bigO$-linear morphisms that are compatible with $\varphi_i$ and $\Fil^i$ structures.
We denote $\MFtor$ the full sub-category of $\bigMFO$ consisting of objects which are killed by some $p$-power, and denote $\MFfree$ the full category of $\bigMFO$ whose objects are finite free over $W(k)_{\bigO}$. Obviously, if $M \in \MFfree$ then $M/ \varpi^m M$ is in $\MFtor$ for all $m$.

It turns out that the category $\MFtor$ is abelian (see \S1.10 in \cite{FL82}).  An object $M$ in $\MFtor$ is called \emph{nilpotent} if there is no nontrivial subobject  $M'\subset M$ such that $\Fil^{1}M'=\{0\}$. Denote the full subcategory of nilpotent objects by $\MFtor^\n$. An object $M \in \MFfree$ is called \emph{nilpotent} if $M/ \varpi^m M$ is nilpotent for all $m$. Denote by $\MFfree^\n$ the full subcategory of $\MFfree$ whose objects are nilpotent.

We refer readers to \cite{fo3} for the construction  and details of the period ring $\Acris$ (and $\Acris$ is just $S$ in \cite{FL82}). Here we just recall that $\Acris$ is a $W(\bar k)$-algebra with a decreasing filtration of ideals $\Acris = \Fil ^0 \Acris \supset   \Fil ^1 \Acris \supset \ldots$, a continuous ring endomorphism $\varphi$ which extends Frobenius on $W(\bar k)$ and a continuous $G_{K_0}$-action which commutes with $\varphi$ and preserves $\Fil ^i \Acris$. It turns out that $\varphi (\Fil ^i \Acris ) \subset p^i \Acris$ for $1 \leq i \leq p-1$ and we define maps $\varphi_i : = \varphi/{p^i} : \Fil ^i \Acris \to \Acris$.
Let $\Rep_{\bigO}(G_{K_0})$ be the category of finitely generated $\bigO$-modules with  continuous $\bigO$-linear $G_{K_0}$-action. We define a  functor $T_\cris^*$ from the category $\MFtor ^\n$ (resp. $\MFfree^\n$) to $\Rep_{\bigO}(G_{K_0})$:
$$T_\cris^* (M): = \Hom_{W(k), \varphi_i, \Fil ^i} (M, \Acris \otimes_{\Z_p} (\Q_p/\Z_p)) \text{ if } M \in \MFtor , $$
 and
$$T_\cris^* (M) := \Hom_{W(k), \varphi_i, \Fil ^i} (M, \Acris) \text{ if } M \in \MFfree.$$
Let $\repcris$ denote the category of continuous $E$-linear $G_{K_0}$-representations on finite dimensional $E$-vector spaces $V$ which are crystalline with Hodge-Tate weights in $\{0, \ldots, p-1\}$.
An object $V \in \repcris$ is called \emph{nilpotent} if  $V$ does not admit nontrivial unramified quotient (it is easy to check that $V$ admits a nontrivial unramified quotient as a $\Q_p$-representation if and only if $V$ admits a nontrivial unramified quotient as an $E$-representation. See the proof of Theorem \ref{equiv} (4) below). We denote by $\repcrisOu$ the category of $G_{K_0}$-stable $\bigO$-lattices in nilpotent representations in $\repcris$.

We gather the following useful results from \cite{FL82} and \cite{Laf80}.

\begin{thm}\label{equiv}\begin{enumerate}

\item The contravariant functor $T^*_\cris$ from $\MFtor^\n$ to $\Rep_\bigO (G_{K_0})$ is exact and fully faithful.
\item An object $M \in \MFfree$ is nilpotent if and only if $M/ \varpi M$ is nilpotent.
\item The essential image of $T^*_\cris: \MFtor^\n \to \Rep_{\bigO}(G_{K_0})$ is closed under taking sub-objects and quotients.
\item Let $V$ be a crystalline representation of $G_{K_0}$ and $K'$ a finite unramified extension of $K_0$. Then $V$ is nilpotent if and only if $V|_{G_{K'}}$ is nilpotent.
\item $T^*_\cris$ induces an anti-equivalence between the category $\MFfree^\n $ and the category $\repcrisOu$.

\end{enumerate}
\end{thm}
\begin{proof}(1) and (2) follow from Theorem 3.3 and Theorem 6.1 in \cite{FL82}. Note that $\underline{\text{U}}_S$ in \cite{FL82} is just $T^*_\cris$ here.  To prove (3), we may assume that $\bigO= \Z_p$ and it suffices to check that $T^*_\cris$ sends simple objects in $\MFtor^\n$ to simple objects in $\Rep_\bigO(G_{K_0})$ (see Property 6.4.2 in \cite{Car06}). And this is proved in \cite{FL82}, \S6.13 (a).  (4) is clear because  $V$ is nilpotent if and only if $(V^*)^{I_{K_0}} = \{0\}$ where $I_{K_0}$ is the inertia subgroup of $G_{K_0}$.

(5) has been essentially proved in \cite{FL82} and \cite{Laf80} but has not been recorded in literature. So we sketch the proof here. First, by \S7.14 of \cite{FL82}, $T^*_\cris(M) $ is a continuous $\bigO$-linear $G_{K_0}$-representation on a finite free $\bigO$-module $T$. By (1) and Theorem 0.6 in \cite{FL82}, we have $\text{rank}_{\bigO} (T) = \text{rank}_{W(k)_\bigO} M= d$. It is easy to see that $M$ is a $W(k)$-lattice in $D_\cris (V^*)$ where $V = \Q_p \otimes_{\Z_p} T$. Hence $V$ is crystalline with Hodge-Tate weights in $\{0, \dots, p-1\}$. To see that $V$ is nilpotent, note that $V$ has an unramified quotient  $\tilde V$ is equivalent to  that there exists an $M'\subset M$ such that
$M'\cap \Fil ^1 M = \{0\}$ and $M/M'$ has no $p$-torsion (just let $M ':= D_\cris (\tilde V^*) \cap M$). So $M$ is nilpotent implies  that $V$ is
nilpotent. Hence by (1), $T^*_\cris$ is an exact, fully faithful functor from  $\MFfree^\n $ to $\repcrisOu$.

To prove the essential surjectivity of $T^*_\cris$, it suffices to assume that $\bigO= \Z_p$. Indeed, suppose that $T$ is an object in $\repcrisOu$ with $d= \text{rank}_{\bigO} T$. Let $V = \Q_p \otimes_{\Z_p}T$ and $D= D_\cris (V^*)$. It is well-known that $D$ is a finite free
$E \otimes_{\Q_p}K_0$-module with rank $d$.  If there exists an $M\in {\mathcal{MF}}^\n _{\Z_p, \text{fr}}$ such that $T_\cris ^*(M) \simeq T$ as $\Z_p[G]$-modules. By the full faithfulness of $T^*_\cris$, $M$ is naturally a $W(k)_\bigO$-module. Since $D$ is $E \otimes_{\Q_p} K_0$-free,  it is standard to show that $M$ is finite $W(k)_\bigO$-free  by computing $\bigO_i$-rank of $M_i$, where $M_i: = M \otimes_{W(k)_\bigO} \bigO_i$ and  $W(k)_\bigO\simeq \prod_i \bigO_i$.

Now suppose that $V \in \textnormal{Rep}_{\Qp, \textnormal{cris}}^{[0, p-1]}(G_{K_0})$ is nilpotent and $D= D_\cris (V^*)$. By \cite{Laf80} \S3.2, there always exists a $W(k)$-lattice $M \in \mathcal{MF}_{\Z_p, \text{fr}}$ inside $D$.  We claim that $M$ is nilpotent. Suppose otherwise, then $\bar M :=M/ p M$ is not nilpotent, and there exists $N \subset \bar M$ such that $\Fil ^1 N =\{ 0\}$. Consequently $\varphi_0 (\Fil ^0 N) = \varphi_0 (N) = N$. Thus $\bigcap_m (\varphi_0)  ^m  (M) \not = \{0\}$. By Fitting Decomposition Theorem, we see that  $M^{\textnormal{mult}}: =\bigcap_m (\varphi_0)  ^m  (M) \not = \{0\}$ is in fact a direct summand of $M$. Let $D^{\textnormal{mult}}=M^{\textnormal{mult}}\otimes_{W(k)}K_0$, it is a $\varphi$-submodule of $D$. Since $D$ is weakly admissible, $t_H(D^{\textnormal{mult}}) \leq t_N(D^{\textnormal{mult}})=0$. Thus we must have $t_H(D^{\textnormal{mult}})=t_N(D^{\textnormal{mult}})=0$, and $D^{\textnormal{mult}}$ is weakly admissible. It is clear that $V^*_{\textnormal{cris}}(D^{\textnormal{mult}})$ is an unramified quotient of $V$, contradicting that $V$ is nilpotent. Thus, $M$ is nilpotent.

It remains to show that any $G_{K_0}$-stable $\Z_p$-lattices $L' \subset V$ is given by an object $ M ' \in {\mathcal{MF}}^\n _{\Z_p, \text{fr}}$. Let $L: = T^*_\cris (M)$. Without loss of generality, we can assume that $L' \subset L$. For sufficiently large $m$, $p^mL \subset L'$, so $L'/p^mL \subset L/p^mL$. Since $L/ p^m L \simeq T^*_\cris (M / p^m M)$. By (3), there exists an object $M'_m \in \MFtor^\n$ such that $T^*_\cris (M'_m)  \simeq L'/ p^m L $. Finally $M' = \varprojlim_m  M'_m$ is the desired object in $\MFfree^\n.$

\end{proof}

Contravariant functors like $T_\cris^*$ are not convenient for deformation theory. So we define a covariant variant for $T_\cris^*$. Define $T_\cris (M) : = (T^*_\cris (M))^*(p-1)$, more precisely,
$$T_\cris (M):= \Hom_{\bigO}( T^*_\cris (M), E/ \bigO)(p-1)\text{ if } M \in \MFtor^\n ,  $$
and
$$T_\cris (M):= \Hom_{\bigO}( T^*_\cris (M),  \bigO)(p-1)\text{ if } M \in \MFfree^\n .  $$

Let $\rho: G_{K_0} \to \GL_d (\bigO)$ be a continuous representation such that there exists an $M \in \MFfree^\n$ satisfying $T_\cris (M) =  \rho$. Then
$T_\cris (\bar M)= \bar \rho:= \rho \mod \varpi \bigO$ where $\bar M: =M / \varpi M $.
Let $\mathcal C^f_\bigO$ denote the category of Artinian local
$\bigO$-algebras for which the structure map $\bigO \to  R$ induces an isomorphism
on residue fields. The morphisms in the category are local homomorphisms inducing isomorphisms on the residue fields. Define a deformation functor
$$D_\cris^\n (R):= \{\text{lifts } \tilde \rho: G_{K_0} \to \GL_d (R) \text{ of } \bar \rho   |  \exists M \in \mathcal{MF}_{\bigO,  \textnormal{tor}}^\n \text{ satisfying } T_\cris (M) \simeq \tilde \rho\}.$$
Here $T_\cris (M) \simeq \tilde \rho$ as $\bigO[G_{K_0}]$-modules. To recapture the $R$-structure, let $\mathcal{MF}_{R}$ be the category similarly defined as $\bigMFO$ by changing $\bigO$ to $R$ everywhere (morphisms in $\mathcal{MF}_{R}$ are $W(k)_R$-morphisms). It is clear that $\mathcal{MF}_{R}$ is a subcategory of $\mathcal{MF}_{\bigO,  \textnormal{tor}}$ (note that $R$ is a $p$-power torsion ring).
Let $\mathcal{MF}^\n_{R, \text{fr}}$ be the full subcategory of $\mathcal{MF}_{R}$ whose objects are nilpotent (as objects in $\mathcal{MF}_{\bigO,  \textnormal{tor}}$) and finite $W(k)_R$-free, and $\Rep_{R, \text{fr}}(G_{K_0})$ the category of $R$-linear continuous representations of $G_{K_0}$ on finite free $R$-modules. It is easy to show that $T_\cris$ restricted to $\mathcal{MF}^\n_{R, \text{fr}}$ is an exact fully faithful functor from $\mathcal{MF}^\n_{R, \text{fr}}$ to $\Rep_{R, \text{fr}}(G_{K_0})$. Thus, the $R$-structure on $\tilde \rho$ guarantees an $R$-structure on $M$ in the definition of $D_\cris ^\n(R)$, i.e., if $T_\cris (M) \simeq \tilde \rho$, then $M \in \mathcal{MF}^\n_{R, \text{fr}}$.

\begin{prop}\label{def}
Assume that $K_0 \subset E$. Then $D^\n_\cris$ is pro-represented by a formally smooth $\bigO$-algebra $R^{\n} _{\bar \rho, \cris}$.
\end{prop}
\begin{proof} By (1) and (3) in Theorem \ref{equiv}, and by \S1 in \cite{Ram93}, $D^\n_\cris$ is a sub-functor of the framed Galois deformation functor of $\bar \rho$ and pro-represented by an $\bigO$-algebra $R^{\n}_{\bar \rho ,\cris}$. The  formal smoothness of $R^{\n}_{\bar \rho, \cris}$ can be proved similarly as in Lemma 2.4.1 in \cite{CHT08}. Indeed, suppose that $R$ is an object of $\mathcal C^f_\bigO$ and $I$ is an ideal of
$R$ with $\frak m_RI = (0)$. To prove the formal smoothness of $R^{\n}_{\bar \rho, \cris}$, we have  to show that any lift in $ D^\n_\cris (R/ I) $ admits a lift in $D^\n _\cris (R)$. Then this is equivalent to lift the corresponding $N \in \mathcal{MF}_{R/I,  \textnormal{fr}}$ to $\tilde N \in  \mathcal{MF}_{R,  \textnormal{fr}}$ (note that any lift $N$ of $\bar M$ will be automatically in $\mathcal{MF}_{R, \text{fr}}^\n$ by Theorem 6.1 (i) of \cite{FL82} or Theorem \ref{equiv} (3)). The proof is verbatim as in Lemma 2.4.1 in \cite{CHT08}. Note that the proof did not use the restrictions (assumed for \S 2.4.1 in loc. cit.)  that $\Fil^{p-1}M = \{0\}$ and
$ \dim _k (\text{gr} ^i {{\bf G}}^{-1}_{\tilde v}(\bar r |_{G_{F_{\tilde v }}})) \otimes_{\bigO_{F_{\tilde v }}, \tilde \tau  } \bigO \leq 1 .$
\end{proof}

\section{The Main Theorem and Its Proof}
\begin{thm} [\textbf{Main Theorem}]
Suppose $\rho: G_{K_0} \to \GL_d(\Qpbar)$ is a crystalline representation, and for each $\tau: K_0 \hookrightarrow \Qpbar$, the Hodge-Tate weights  $\HT_{\tau}(\rho) \subseteq \{a_{\tau}, \ldots, a_{\tau}+p-1\}$ for some $a_{\tau}$, then $\rho$ is potentially diagonalizable.
\end{thm}

\begin{proof}
By Lemma 2.2.1.1 of \cite{BM02}, we may assume that $\rho$ factors through $\GL_d (\bigO)$ for a sufficiently large $\bigO$.
By Lemma \ref{reduce to irreducible}, we can assume that $\rho$ is irreducible and hence  $\rho^* (p-1)$ is nilpotent. As in the proof of Lemma 1.4.3 of \cite{BLGGT10}, twisting by a suitable crystalline character, we can assume $a_{\tau}=0$ for all $\tau$. Then we can choose an unramified extension $K'$, such that $\rhobar\mid_{G_{K'}}$ has a $G_{K'}$-invariant filtration with 1-dimensional graded pieces. By Theorem \ref{equiv} (4), $\rho^*(p-1)$ is still nilpotent when restricted to $G_{K'}$. Without loss of generality, we can assume that $K _0 = K'$.
Now by Theorem \ref{equiv} (5), there exists an $M \in \MFfree^\n$ such that $T_\cris (M) \simeq \rho$. Then $\bar M := M / \varpi M$ is nilpotent and $T_\cris (\bar M) \simeq \bar \rho$. Note that $\bar M$ has a filtration with rank-1 $ k \otimes_{\Z/p\Z}\mathbb F$-graded pieces to correspond to the filtration of $\bar \rho$. Now by  Lemma 1.4.2 of \cite{BLGGT10}, we can lift $\bar M$ to $M'\in \MFfree$ which has filtration with  rank-1 $W(k)_\bigO$-graded pieces (note that the proof of Lemma 1.4.2 of \cite{BLGGT10} did not use the restriction that $\HT_{\tau}(\rho) \subseteq \{0, \ldots, p-2\}$). Hence $M'$ is nilpotent by Theorem \ref{equiv} (2). Then $\rho'=  T_\cris (M')$ is crystalline and has a $G_{K_0}$-invariant filtration with 1-dimensional graded pieces by Theorem \ref{equiv} (5).  Then part 1 of Lemma 1.4.3 of \cite{BLGGT10} implies that  $\rho'$ is potentially diagonalizable. Now it suffices to show that $\rho $ connects to $\rho'$. But it is obvious that $R^\n_{\bar \rho, \cris}$ is a quotient of $R^{\Box}_{\bar \rho, \{\HT_{\tau}(\rho)\}, K\textnormal{-cris}}$. By Proposition \ref{def}, we see that $\rho$ and $\rho'$ must be in the same connected component of $\Spec(R^{\Box}_{\bar \rho, \{\HT_{\tau}(\rho)\}, K\textnormal{-cris}}[\frac 1 p])$. Hence $\rho\sim \rho'$ and $\rho$ is potentially diagonalizable.

\end{proof}

\bibliographystyle{alpha}

\end{document}